\documentclass[fleqn]{article}
\usepackage[utf8]{inputenc}
\usepackage{enumerate,amsthm,amssymb,xcolor,hyperref,comment}

\usepackage[leqno]{amsmath}
\usepackage{graphicx}

\makeatletter
\newcommand{\leqnomode}{\tagsleft@true}
\newcommand{\reqnomode}{\tagsleft@false}
\makeatother

\usepackage{latexsym}
\usepackage{amsfonts}
\usepackage[leqno]{amsmath}
\usepackage{float}
\usepackage{lmodern}

\usepackage{marvosym}

\begin{document}

\title{Attempting perfect hypergraphs}
\author{Maria Chudnovsky\thanks{This material is based upon work supported in
    part by the U. S. Army
Research Office under grant   number W911NF-16-1-0404, and supported by  NSF grant DMS 1763817 and ERC advanced grant 320924.}\\
Princeton University, Princeton, NJ 08544, USA
\\
\\
and Gil Kalai\thanks{Supported by ERC advanced grant 320924 and by ISF Grant 1612/17 .}\\
Reichman University, Herzliya 4610101, Israel\\
and The Hebrew University of Jerusalem, Jerusalem 91904, Israel}
 
\date {June 7, 2020; revised \today}

\newtheorem{theorem}{}[section]

\maketitle
\abstract
    {

%\begin {abstract}
    
      We study several extensions of the notion of perfect graphs to $k$-uniform hypergraphs.
      One main definition extends to hypergraphs the notion of perfect graphs based on coloring. Let $G$ be a $k$-uniform hypergraph.
A coloring of a $k$-uniform hypergraph $G$ is {\em proper} if it is a coloring of the $(k-1)$-tuples with
elements in $V(G)$ in such a way that  no edge of $G$ is a monochromatic
$K_k^{k-1}$.

A $k$-uniform hypergraph $G$ is {\it $C_\omega$-perfect} if 
for every induced subhypergraph $G'$ of
$G$ we have:
\begin{itemize}
  \item if  $X \subseteq  V(G')$ with $|X|<k-1$, then  
there is a proper $(\omega(G')-k+2)$-coloring of $G'$ (so $(k-1)$-tuples are
colored) that
restricts to a proper $(\omega(G')-k+2)$-coloring of $lk_{G'}(X)$ (so
$(k-|X|-1)$-tuples are colored).
\end {itemize}
Another main definition is the following: 
A $k$-uniform hypergraph $G$ is hereditary perfect (or, briefly, $H$-perfect) if all links of sets of $(k-2)$ vertices
are perfect graphs. 

The notion of $C_\omega$ perfectness is not closed under complementation (for $k>2$) and we define $G$ to be doubly perfect if both $G$ and its complement are $C_\omega$ perfect.
We study doubly-perfect
hypergraphs: In addition to
perfect graphs nontrivial doubly-perfect graphs consist of a restricted interesting class of 3-uniform hypergraphs, and within this class we give a complete characterization of doubly-perfect $H$-perfect hypergraphs.  

}
%\end {abstract}

      %
%      We define a $k$-uniform hypergraph $G$ to be weakly perfect if the link of sets of $k-2$ vertices are perfect graphs.   
%      $G$ is $\omega$-perfect if it is weakly perfect and, in addition, every $k+1$ vertices contains either $k+1$ edges or at most two edges.
%      $G$ is $\alpha$-perfect if its complement is $\omega$-perfect.
%      We give some basic examples, equivalent formulations in terms of colorings, and discuss simple extremal problems.
%      We characterize weakly perfect 3-uniform cocycles, and note that
%      they are precisely the family of the non trivial $k$-uniform hypergraphs ($k>2$) which are both $\omega$-perfect and $\alpha$-perfect.

      %      $G$ is $\alpha-perfect$ if in addition, every $r+1$ vertices contains either no edges or at most $r-1$ edges.
%      $G$ is $R$-perfect if for every induced subhypergraph $H$, $|V(H)| \le R_{k-1}(\alpha(H), \omega(H))$.
%      $G$ is strongly perfect if it is both $G$ and its complement $G^C$ are $\omega$-perfect.

\section{Introduction}

The purpose of this paper is to study some extensions of the notion of perfect graphs to $k$-uniform hypergraphs, $k>2$.
We start with basic notation and terminology regarding hypergraphs
followed by basic definitions and properties of perfect graphs.
Denote by $K^p_q$ the complete $p$-uniform hypergraph on $q$ vertices.
Let $G$ be a $k$-uniform hypergraph.
The {\em complement}
$G^c$ of $G$ is the $k$ uniform hypergraph such that a $k$-subset $X$ of
$V(G)$ is a hyperedge of $G^c$ if and only if $X$ is not a hyperedge of $G$.
For $X \subseteq  V(G)$ we denote by $G[X]$ the $k$-uniform  hypergraph
induced by $G$ on $X$, and by $lk_G(X)$ the $(k-|X|)$-uniform
hypergraph with vertex set $V(G) \setminus X$ and such that $Y \in E(G(X))$ if
and only if $X \cup Y \in E(G)$ (this is the {\em link} of $X$).
Observe that $lk_G^c(X)=lk_{G^c}(X)$. We denote by
$\omega(G)$ the maximum size of $X \subseteq V(G)$ such that
$G[X]$ is a complete $k$-uniform hypergraph, and by
$\alpha(G)$ the maximum size of $X \subseteq V(G)$ such that
$G^c[X]$ is a complete $k$-uniform hypergraph.

\begin {figure}
%\centering
\includegraphics[scale=0.3]{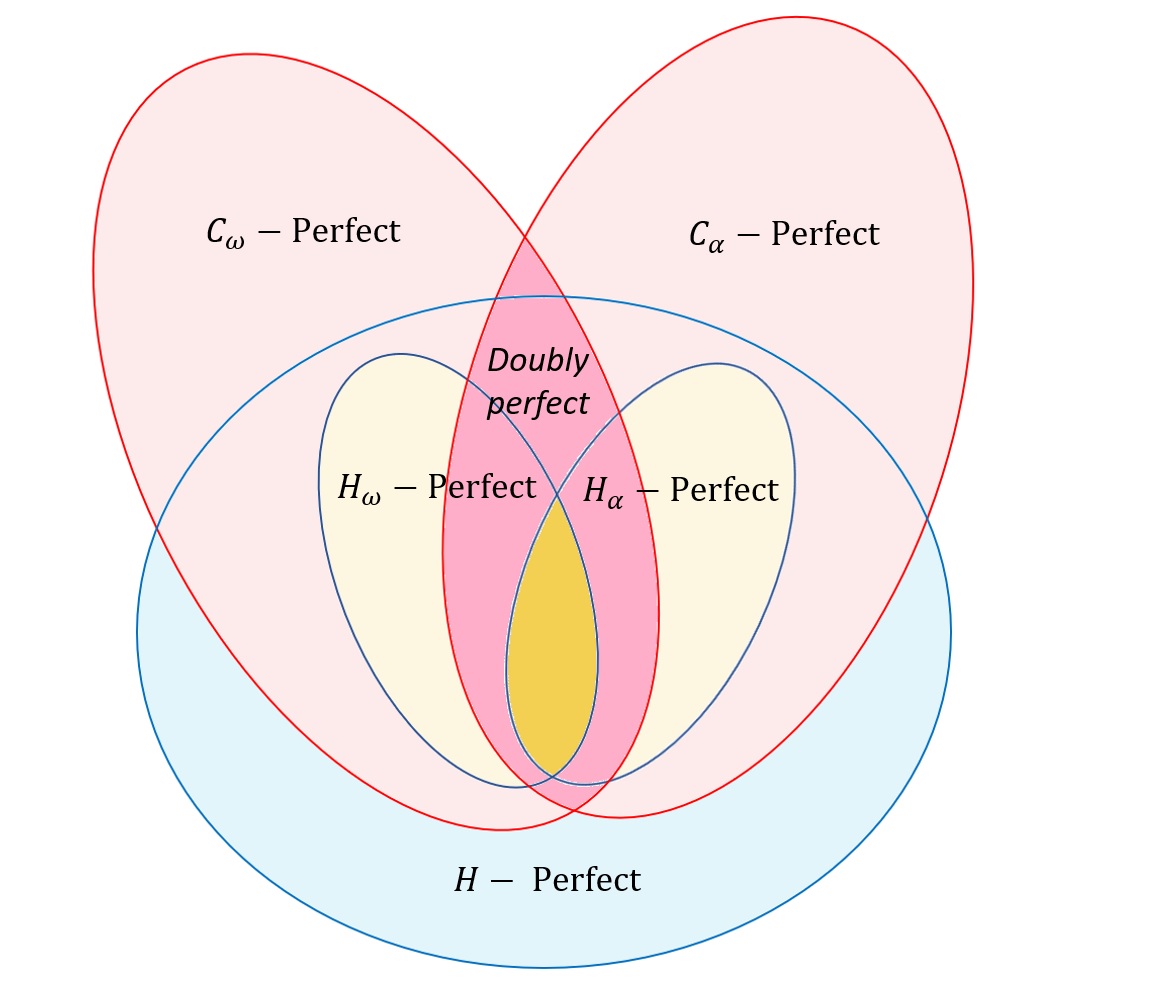}
\caption{Classes of perfect hypergraphs considered in this paper. The figure describes (out of scale) the case of 3-uniform hypergraphs.
  For graphs
  all the classes coincide with the class of perfect graphs. For $k$-uniform hypergraphs, $k>3$ the class
  of doubly-perfect hypergraphs
  contains
  only the empty hypergraph and complete hypergraph.
%  The blue elipse represents entropy splitting hypergraphs.
}
\label{fig:perfect}
\end{figure}

%Let $G$ be a $k$-uniform hypergraph. We say that $X \subseteq V(G)$ is
%connected if there is a sequence of edges $e_1, \ldots, e_p$ of $G$ such that
%$X=\bigcup_{i=1}^p e_i$ and 
%for every $i$ we have $|e_i \setminus \bigcup_{j=1}^{i-1} e_j|=1$. 
%$e_1,e_2 \in E(G)$ are
%{\em $t$-adjacent} if $|e_1 \cap e_2|=t$. A {\em $t$-path} in $G$ is a
%sequence $e_1, \ldots, e_p$ of edges of $G$ where $e_i$ is $t$-adjacent
%to $e_{i+1}$. We say that $X \subseteq V(G)$ is $t$-connected if for
%every $u,v \in X$ there  is a $t$-path $e_1, \ldots, e_p$ such that
%$u \in e_1$ and $v \in e_p$.

A graph $G$ is {\it perfect} if for every induced subgraph $H$ of $G$, $\chi (H)=\omega(H)$. The Weak Perfect Graph Theorem \cite {Lov72}
asserts that $G$ is perfect iff its complement $G^c$ is perfect. The Strong Perfect Graph Theorem \cite {CRST06} asserts
that $G$ is perfect iff $G$ is
a {\it Berge} graph, namely neither $G$ nor $G^c$  contains an induced cycle of odd length greater than 3. The class of perfect graphs
is a remarkable class of graphs with profound connections in mathematics, theoretical computer science and optimization. 
There is also a rich area of graph theory devoted to the study of graphs with forbidden induced subgraphs, and perfect graphs
play a central role in this study. 

Our purpose here is to propose several extensions of the notion of perfect graphs to $k$-uniform hypergraphs, $k>2$.
%One nice property of our class of Perfect graphs is that the equivalence between Berge graphs and perfect graphs extends to $k>2$.
%We note that other possible extensions of the notion of perfect graphs to hypergraphs might be possible 
%\section{Berge and Perfect}
%$G^c(X)$ is the complement of
%$G(X)$ (when viewed as a $k-|X|$-uniform hypergraph), so the notation is not
%misleading. Sometimes we write $G(X)$ to mean $lk_G(X)$.
Our main definition extends to hypergraphs the notion of perfect graphs based on coloring.
%In Section \ref {s:coloring} we give an equivalent
%definition of $\omega$-perfect graphs in terms of
%proper coloring. 
%We will give now an alternative definition of $\omega$-perfectness based on coloring:
Let $G$ be a $k$-uniform hypergraph.
A coloring of $G$ is {\em proper} if it is a coloring of the $(k-1)$-tuples with
elements in $V(G)$ in such a way that  no edge of $G$ is a monochromatic
$K_k^{k-1}$.

%{\bf Definition of $C$-perfect}:
A $k$-uniform hypergraph is {\it $C_\omega$-perfect} if 
for every induced subhypergraph $G'$ of
$G$ we have:
\begin{itemize}
  \item if  $X \subseteq  V(G')$ with $|X|<k-1$, then  
there is a proper $(\omega(G')-k+2)$-coloring of $G'$ (so $(k-1)$-tuples are
colored) that
restricts to a proper $(\omega (G')-k+2)$-coloring of $lk_{G'}(X)$ (so
$(k-|X|-1)$-tuples are colored).
\end{itemize}

%\end {theorem}

In Section \ref {s:coloring} 
we show that a weaker coloring
property which resembles the Berge property for graphs suffices for $C_\omega$-perfectness.
%Section present also some basic examples, of $\omega$-perfect hypergraphs along with some simple extremal results for them.
Define next $G$ to be {\em $C_\alpha$-perfect} if its complement $G^c$ is $C_\omega$-perfect. Now, call
%{\bf Definition of doubly perfect}:
$G$  {\it doubly perfect} if $G$ is both $C_\omega$-perfect and $C_\alpha$-perfect. 
We study the class of doubly-perfect hypergraphs in Section \ref {s:co}. 
A $k$-uniform hypergraph is {\em clique friendly} if
every set of  $k+1$ vertices contains either $k+1$ edges or at most two edges.
It follows easily from the definition that $C_\omega$-perfect hypergraphs are clique friendly and this implies that
if $G$ is a doubly-perfect hypergraph which is neither complete nor
empty,  then $k \le 3$. Moreover, for $k=3$ to be doubly perfect $G$ must be
 a cocycle, namely every 4 vertices span an even number of edges. 
 An equivalent definition of cocycles which we rely on in Section \ref {s:co} is the following: for a graph $G$ we write $co(G)$ to be 
 the 3-uniform hypergraph whose edges correspond to triples of vertices of $G$ that span an odd number of edges. 
 Every 3-uniform cocycle $C$ can be written 
 as $co(G)$ for some graph $G$.

%{\bf Definition of weakly perfect}:
A $k$-uniform hypergraph $G$ is {\it $H$-perfect} (or hereditary perfect) if the link $lk_G(X)$
of every set $X$ of $k-2$ vertices is a perfect graph.   
%Let us now  introduce our new definitions of perfect and Berge hypergraphs.
%we prove that $G$ is $C$-perfect iff it is $\omega$-perfect
%{\bf Definition of $\omega$-perfect}:
%A $k$-uniform hypergraph is {\em clique friendly} if
%every set of  $k+1$ vertices contains either $k+1$ edges or at most two edges.
A $k$-uniform hypergraph is {\it $H_\omega$-perfect} if it is $H$-perfect and
clique friendly.
In Section \ref {s:coloring} we also prove that $H_\omega$-perfect hypergraphs are $C_\omega$-perfect. (The converse is not true.) 
Examples of $H_\omega$-perfect (hence also $C_\omega$-perfect) hypergraphs include $k$-partite hypergraphs,
simple hypergraphs and hypergraphs of
$k$-cliques of perfect graphs.
%More examples are given in Section \ref {s:ex} where we also discuss simple extremal
%problems for perfect hypergraphs.
In Section \ref {s:co}
we also give a full description of
doubly-perfect hypergraphs which are also $H$-perfect. 

Our notion of doubly-perfect hypergraphs is closely related (and yet not identical)
to Simonyi's notion of
``entropy splitting hypergraphs'' \cite {Sim95}. This is also discussed in Section \ref {s:co}. 

Our concluding Section \ref {s:o} discusses other related notions of perfectness and possible connections.
We note that a different notion of perfectness for hypergraphs was pioneered by Voloshin \cite {Vol95,Vol02}
and further studied by Bujt\'as and Tuza \cite {BujTuz09}.
Seeking hypergraph analogs of
perfect graphs fits the ‘high-dimensional combinatorics’ programme of Linial \cite{Lin08,Lin18}. 

\section {Perfect hypergraphs and proper colorings}
\label {s:coloring}
\subsection {Basic properties of $C_\omega$ perfect hypergraphs}

%For a  $k$-uniform hypergraph $G$
%The following are equivalent

%(i) %is {\em $\omega$-perfect} if
%For every induced subhypergraph $G'$ of
%$G$ we have:
%\begin{itemize}
%  \item if  $X \subseteq  V(G')$ with $|X|<k-1$, then  
%there is a proper $(\omega(G')-k+2)$-coloring of $G'$ (so $(k-1)$-tuples are
%colored) that
%restricts to a proper $(\omega(G')-k+2)$-coloring of $lk_{G'}(X)$ (so
%$(k-|X|-1)$-tuples are colored).
%\end{itemize}

%(ii) $G$ is $\omega$-perfect.
%\end {theorem}

%\begin {theorem}

% In this subsection we discuss the equivalence of $C$-perfectness and $\omega$-perfectness. We start with two lemmas. 

    \begin {theorem} \label {tetra}
      Let $G$ be a $C_\omega$-perfect $k$-uniform hypergraph with $|V(G)|=k+1$.
      Let $X \subseteq V(G)$ with $|X|<k-1$, and suppose
      that $lk_G(X)$ is a clique. Then $G$ is a clique.
      Consequently, every $C_\omega$-perfect hypergraph is clique friendly.
     \end{theorem}
%{\bf Proof of claim 1:}
%{\bf Maria: I am not sure we use this anymore, but it's kind of a nice fact}
    \begin {proof}
  Suppose $G$ is not a clique, and so $\omega(G)=k$. Let
  $a,b,c \in V(G) \setminus X$. 
  Since $G$ is $C_\omega$-perfect,
  there exists a proper $2$-coloring of $G$ (so the $(k-1)$-tuples are colored)
  that restricts to a proper $2$-coloring of the complete graph with vertex set
  $\{a,b,c\}$ (so the vertices are colored), a contradiction.
  This proves \ref{tetra}.
\end{proof}

\begin {theorem} \label {CFwtetra}
      Let $G$ be a clique-friendly $k$-uniform hypergraph with $|V(G)|=k+1$.
      Let $X \subseteq V(G)$ with $|X|<k-1$, and suppose
    that $lk_G(X)$ is a clique. Then $G$ is a clique.
     \end{theorem}
%{\bf Proof of claim 1:}
\begin {proof}
  Suppose $G$ is not a clique. Then $|E(G)|<k+1$. Since $lk_G(X)$ is a clique,
  and $|V(G)|-|X| \geq 3$, we have that $|E(G)| \geq 3$, contrary to the fact
  that $G$ is clique friendly.
    This proves \ref{CFwtetra}.
\end{proof}

\begin{theorem} \label{friendlycliques}
  If $G$ is clique friendly then for every $X \subseteq V(G)$ with
  $|X|<k-1$, and every clique $K$ in $lk_G(X)$, we have that
  $K \cup X$ is a clique of $G$.  Consequently,
  $\omega(lk_G(X))  \leq \omega(G)-|X|$.
  %, and every clique of $G$ containing
  %$X$ has size $\omega(lk_G(X))+|X|$.
\end{theorem}

\begin{proof}

First we prove the first statement.
  Induction on $|X|$.
  Let $K$ be a clique in $lk_G(X)$. 
  Assume first that $X=\{x\}$. By \ref{CFwtetra}, $K \cup \{x\}$ is a clique of
  $G$, as required.

  Now let $x \in X$; write $X'=X \setminus \{x\}$.
  Then $K$ is a clique in $lk_G(x)(X')$, and so
  inductively $K \cup X'$ is a clique in $lk_G(x)$.
  But   also, as we have seen in the base case of the induction,
  $(K \cup X') \cup \{x\}$ is a clique in $G$, as required. This proves the
  first  statement.

  To prove the second statement, let $K$ be a clique in $lk_G(X)$ of size
  $\omega(lk_G(X)$. We have proved that $K \cup X$ is a clique of $G$,
  and consequently $\omega(G) \geq \omega(lk_G(X))+|X|$.

  This proves \ref{friendlycliques}.
\end{proof}

Now \ref {tetra} and \ref{friendlycliques} imply:

\begin{theorem} \label{smallerclique}
  If $G$ is $C_\omega$-perfect then for every $X \subseteq V(G)$ with $|X|<k-1$, $\omega(lk_G(X)) \leq \omega(G)-|X|$.
\end{theorem}

\subsection {Berge is perfect}

%{\bf Gil's comment: Maybe we should reconsider the word ``Berge'' also, more importantly I am not sure about how the proof goes}

A  $k$-uniform hypergraph is {\it Berge} if
for every induced subhypergraph $G'$ of $G$ we have:

\begin{itemize}
\item If  $X \subseteq  V(G')$ with $|X|=k-2$, then  
there is a proper $(\omega(G')-k+2)$-coloring of $G'$ (so $(k-1)$-tuples are
colored) that
restricts to a proper $(\omega(G')-k+2)$-coloring of $lk_{G'}(X)$ (so
$(k-|X|-1)$-tuples are colored).
\end{itemize}

\begin {theorem}
\label {t:berge}
%A weakly perfect $k$-uniform hypergraph is Berge iff it is $C_\omega$-perfect.
%Berge $k$-uniform hypergraphs is Berge iff it is $C_\omega$-perfect.
A $k$-uniform hypergraph $G$  is Berge iff it is $C_\omega$-perfect.
\end {theorem}

\begin {proof}

%  Clearly if $G$ is $C_\omega$-perfect then it Berge. When we follow carefully the proof
%  of the implication $C$-perfect implies $\omega$ perfect
%  we notice that we use only the fact that links of sets of size $k-1$ that are graphs are perfect,
%  and the two statements \ref {tetra} and \ref {smallerclique} which only rely on $G$ being Berge.

Clearly if $G$ is $C_\omega$-perfect, then $G$ is Berge.
 Thus we assume that $G$ is Berge and show that $G$ is $C_\omega$-perfect.
  Since being Berge and being $C_\omega$-perfect are both closed under taking
  induced subhypergraphs, it is enough to prove that:

  \begin{itemize}
\item
If  $X \subseteq  V(G)$ with $|X| \leq k-2$, then  
there is a proper $(\omega(G)-k+2)$-coloring of $G$ (so $(k-1)$-tuples are
colored) that
restricts to a proper $(\omega(G)-k+2)$-coloring of $lk_{G}(X)$ (so
$(k-|X|-1)$-tuples are colored).
  \end{itemize}
  
Let $X \subseteq V(G)$. 
If $|X|=k-2$, the statement follows immediately from the first bullet in the
definition of Bergeness. So we may assume that $|X|<k-2$.

Order the vertices of $V(G)$ as $v_1, \ldots, v_n$ so that
$X=\{v_1, \ldots v_{|X|}\}$.
Let $Y \subseteq V(G)$ with $|Y|=k-2$ 
and let $i$ be maximum such that
  $v_i \in Y$.
  Let $R(Y)=lk_G(Y)[v_{i+1}, \ldots, v_n]$.
Then $R(Y)$ is a graph.
  Since $|Y|=k-2$,  the first bullet of the definition of Bergeness implies that
there is a proper $(\omega(G)-k+2)$-coloring $c_Y$ of $G$ (so $(k-1)$-tuples are
colored) that
restricts to a proper $(\omega(G)-k+2)$-coloring of $lk_{G}(Y)$ (so
vertices are colored). Since $R(Y)$ is a subgraph of
$lk_G(Y)$, the coloring $c_Y$ is also a coloring of $R(Y)$.

Let $Z$ be a ($k-1$)-tuple of vertices of $G$, and let $Y$ be the $(k-2)$-initial
segment of $Z$. Define $c(Z)=c_Y(Z\setminus Y )$.

We show that $c$ is proper.
Let $F \in E(G)$, and let  $Y$ be the initial $(k-2)$-segment of $F$.
Then $F \setminus Y$ is an edge $ab$, say,  of $R(Y)$ and so
$c_Y(a) \neq c_Y(b)$. But now $c(Y \cup \{a\}) \neq c(Y \cup \{b\})$, and
so $F$ is not a monochromatic $K_k^{k-1}$ in $c$. 

Next we show that  $c$ restricts to a proper coloring of $lk_G(X)$.
Let  $F \in E(lk_G(X))$, let $Z$ be the initial $(k-2-|X|)$-segment of
$F$, and let $Y=X \cup Z$.
Then $|F \setminus Z|=2$, say $F \setminus Z=\{a,b\}$. Since
  $X=\{v_1, \ldots v_{|X|}\}$,
  it follows that $c(Z \cup \{a\})=c_Y(a)$ and $c(Z \cup \{b\})=c_Y(b)$.
  But $ab$ is an edge of $R(Y)$ and so
$c_Y(a) \neq c_Y(b)$.  Consequently,  $c(Z \cup \{a\}) \neq c(Z \cup \{b\})$, and
therefore $F$ is not a monochromatic $K_{k-|X|}^{k-1-|X|}$ in $c$ when viewed as a
coloring of $lk_G(X)$.  This proves that $G$ is $C_\omega$-perfect.

\end {proof}
We now prove:

\begin {theorem}
  \label {t:main}
%  For a $k$-uniform hypergraph $G$ the following are equivalent:
%  (i) $G$ is $C$-perfect
%  (ii) $G$ is $\omega$-perfect
%  A $k$-uniform hypergraph is $\omega$-perfect if and only if it is $C$-perfect.
An $H_\omega$-perfect hypergraph is $C_\omega$-perfect.
\end {theorem}

  \begin{proof}

    Let $G$ be an $H_\omega$-perfect $k$-uniform hypergraph. 
    By \ref{t:berge} it is enough to prove that  $G$ is Berge.
    We need to show that for every induced subhypergraph $G'$ of
$G$ we have
    \begin{itemize}
\item
If  $X \subseteq  V(G')$ with $|X|=k-2$, then  
there is a proper $(\omega(G')-k+2)$-coloring of $G'$ (so $(k-1)$-tuples are
colored) that
restricts to a proper $(\omega(G')-k+2)$-coloring of $lk_{G'}(X)$ (so
$(k-|X|-1)$-tuples are colored).
\end{itemize}

    Let $G'$ be an induced subhypergraph of $G$, and let $X \subseteq V(G')$ with
    $|X|=k-2$.    
      Order the vertices of $V(G)$ as $v_1, \ldots, v_n$ so that
$X=\{v_1, \ldots v_{|X|}\}$.
Let $Y \subseteq V(G')$ with $|Y|=k-2$ 
and let $i$ be maximum such that
  $v_i \in Y$.
  Let $R(Y)=lk_{G'}(Y)[v_{i+1}, \ldots, v_n]$.
  Then $R(Y)$ is a graph.
  Moreover, $lk_{G'}(Y)$ is an induced subgraph of $lk_G(Y)$,
  and therefore $R(Y)$ is an induced subgraph of 
$lk_G(Y)[v_{i+1}, \ldots, v_n]$.
  Since $|Y|=k-2$ and  $G$ is $H_\omega$-perfect (and therefore $H$-perfect),
  we have that $lk_G(Y)$ is a perfect graph.  It follows that $R(Y)$ is a
  perfect graph.   Consequently, $R(Y)$ has a proper
  $\omega(R(Y)) \leq \omega(lk_{G'}(Y))$-coloring
  $c_Y$.   Since $G$ is $H_\omega$-perfect (and
     therefore
     clique friendly) it follows that $G'$ is clique-friendly, and so
      \ref{friendlycliques} implies 
     that
     $\omega(lk_{G'}(Y))\le \omega(G')-k+2$.

Let $Z$ be a ($k-1$)-tuple of vertices of $G'$, and let $Y$ be the $(k-2)$-initial
segment of $Z$. Define $c(Z)=c_Y(Z\setminus Y )$.

We show that $c$ is proper.
Let $F \in E(G')$, and let  $Y$ be the initial $(k-2)$-segment of $F$.
Then $F \setminus Y$ is an edge $ab$, say,  of $R(Y)$ and so
$c_Y(a) \neq c_Y(b)$. But now $c(Y \cup \{a\}) \neq c(Y \cup \{b\})$, and
so $F$ is not a monochromatic $K_k^{k-1}$ in $c$. 

Since $lk_{G'}(X)=R(X)$,  and the restriction of $c$ to $R(X)$ is $c_X$,
it follows that $c$ restricts to a proper coloring of $lk_{G'}(X)$.
This proves
that $G$ is  Berge. Now \ref{t:main} follows from \ref{t:berge}.
\end{proof}
%Let $G$ be a $k$-uniform hypergraph.
%Recall that a coloring of $G$ is {\em proper} if it is a coloring of the $(k-1)$-tuples with
%elements in $V(G)$ in such a way that  no edge of $G$ is a monochromatic
%$K_k^{k-1}$.

To summarize we have:
\begin{theorem}
  \label{summary}
  %Berge2 $\rightarrow
  $H_{\omega}$-perfect $\rightarrow$ Berge $\rightarrow C_{\omega}$-perfect,
  and $C_{\omega}$-perfect $\rightarrow$ Berge.
\end{theorem}

%\begin{proof}
%%%  %The first implication is \ref{t:berge2}.
%  The first implication is \ref{t:main}.
%  The second and third implications are \ref{t:berge}.
%  \end{proof}
  %We note that if $G$ is weakly perfect then the first condition 

\subsection{Proper coloring and Ramsey numbers}

%Here is a simple connection to Ramsey numbers.
%This is the connection to Ramsey numbers, but maybe  it's not important.
A well-known equivalent definition \cite {Lov72} of perfect graphs (at the heart of the alternative proof of the Weak Perfect Graph
Theorem given in \cite{Gas96}) is the
following key fact:
\begin{theorem}
\label{Gasp}
A graph $G$ is {\em perfect} if and only if for every induced subgraph $G'$ of
$G$ we have that $\alpha(G')\omega(G') \geq |V(G')|$.
\end{theorem}

%Denote by $R_s(k)$ the minimum integer such that for every $n \geq R_s(k)$ in
%every $s$-coloring of $K^{k-1}_n$ there is a monochromatic $K^{k-1}_k$.
%Then $R_s(2)=s+1$, and $G$ is perfect if
%$|V(G')| < R_{\alpha(G') \omega(G')}(2)$
%for every induced subgraph $G'$ of $G$.

We mention yet another related notion of perfect hypergraphs.
Denote by $R_s(k)$ the minimum integer such that for every $n \geq R_s(k)$ in
every $s$-coloring of $K^{k-1}_n$ there is a monochromatic $K^{k-1}_k$.
Then $R_s(2)=s+1$, and $G$ is perfect if
$|V(G')| < R_{\alpha(G') \omega(G')}(2)$
for every induced subgraph $G'$ of $G$.

{\bf Definition of $R$-perfect}:
A $k$ uniform hypergraph $G$ is {\it $R$-perfect} if $|V(G)|<R_{(\alpha(G)-k+2)(\omega(G)-k+2)}(k)$,
and this property holds for all induced subhypergraphs and links.

\begin{comment}
$R$-perfect graphs are weakly perfect. We do not know if every weakly perfect graph is $R$-perfect, and not
even if every $H_\omega $-perfect hypergraph is $R$-perfect.
The connection to proper coloring is seen by the following proposition
(which implies that doubly-perfect hypergraph graphs are $R$-perfect).
\end{comment}

\begin{theorem} \label{perfectR}
  Let $G$ be a $k$-uniform hypergraph. If
  \begin {itemize}
    \item[[PC]]
  $G$ admits a proper coloring with $\omega(G)-k+2$ colors
  and its complement $G^c$ admits a proper coloring
  with $\alpha (G)-k+2$ coloring 
  \end {itemize}
  then 
%Recall that a coloring of $G$ is {\em proper} if it is a coloring of the $(k-1)$-tuples with
%elements in $V(G)$ in such a way that  no edge of $G$ is a monochromatic
%$K_k^{k-1}$.
%  If $G$ is strongly perfect $k$-uniform hypergraph then   
\begin {equation}
  \label {e:ram}
  |V(G)|<R_{(\alpha(G)-k+2)(\omega(G)-k+2)}(k).
\end {equation}  

\end{theorem}

  %This theorem is vacuous for $k>3$
  \begin{proof}
Let $c^G$ be the $(\omega(G)-k+2)$-coloring of $K^{k-1}_{[V(G)]}$ %as in \ref{coloringG}.
and let $c^{G^c}$ be the $(\alpha(G)-k+2)$-coloring of $K^{k-1}_{[V(G)]}$.
%obtained by applying \ref{coloringG}  to $G^c$.
For every $(k-1)$-tuple $e$ let $c(e)=(c^G(e),c^{G^c}(e))$.
Then $c$ is a coloring of $K^{k-1}_{[V(G)]}$ with $(\alpha(G)-k+2)(\omega(G)-k+2)$ colors.

It remains to show that   there is no monochromatic $K^{k-1}_k$ in $c$.
  Suppose that $Y \subseteq V(G)$ is a monochromatic  $K^{k-1}_k$ in $c$.
  If  $Y \in E(G)$, then the coloring is not monochromatic in
  the first coordinate, and if 
  $Y \not \in E(G)$, then the coloring is not monochromatic
    in the second coordinate, a contradiction. This proves \ref{perfectR}.
\end{proof}
  The class of $k$-uniform hypergraphs described by property [PC] for all induced subhypergraphs (but not necessarily links),
  and the  class of $R$-perfect hypergraphs are both self-complementary classes that again, for $k=2$
  consist of the class of perfect graphs.

%{\bf Problem:} Are weakly perfect graphs $R$-perfect?
%{\bf: $R$-perfect graphs are weakly perfect. we do not know if the converse is true.}
%{\bf: A question that we do not know: If $G$ is weakly perfect does  $G$ admits a proper coloring with $\omega(G)-k+2$ colors?}

%{\bf Definition of $R$-perfect}:
%A $k$ uniform hypergraph $G$ is {\it $R$-perfect} if $|V(G)|<R_{(\alpha(G)-k+2)(\omega(G)-k+2)}(k)$,
%and this property holds for all induced subhypergraphs and links.

%%The hypergraph $K^3_4$ minus an edge is not perfect  %shows that for $k > 2$
%The class of $3$-uniform hypergraphs such that
% (\ref {e:ram}) holds for all
%induced subhypergraphs and links includes $K_4^3(-)$ and is thus strictly larger than the
%class of perfect $3$-uniform hypergraphs.

%as seen by 
%Let $v_1-v_2-\cdot -v7-v1$ be an antihole.
%Let $V(H)=\{v_1,..,v_7\}$ and
%$a=\{v_i,v_j,v_k\}$ is an edge of $H$ if a is either a triangle of the antihole, or
%$j=i+1$ and $k=i+2 (\mod 7)$. In this example, the link of $v_1$ a 4-cycle $v_2-v_3-v_6-v_7$ with two additional edges
%$v_6-v_4$ and $v_3-v_5$. 
%%%%One could base a different more inclusive
%%%%definition of perfect $k$-uniform hypergraphs based on equation (\ref {e:ram}).
%{\bf Gil: We need an example of a $k$-uniform hypergraph, $k>2$ that satisfies (\ref {e:ram})
%  but is not perfect. And that the equivalence between Berge and perfect does not hold for a
%  definition of perfectness based on (\ref {e:ram}) for
%  links and induced subcomplexes }

%\section {Examples and extremal properties}
%\label {s:ex}

\subsection {A few examples of $H_\omega$-perfect hypergraphs}

In this subsection we list a few constructions that yield $H_\omega$-perfect $k$-uniform
hypergraphs.
First, a disjoint union of two $H_\omega$-perfect hypergraphs is $H_\omega$-perfect.
Simple 3-uniform hypergraphs are
$H_\omega$-perfect, and so are tripartite 3-uniform hypergraphs.
More generally $k$-uniform hypergraphs, in which no two edges share more than $k-2$ vertices, are $H_\omega$-perfect
and so are $k$-partite $k$-uniform hypergraphs.

A simple induction gives another natural family:

\begin {theorem}
  Let $G$ be a $H_\omega$- perfect $k$-uniform hypergraph, and let $H$ be the hypergraph of cliques
  of size $r$ ($r>k$) in $G$. Then
  $H$ is $H_\omega$-perfect.
  \end {theorem}

%The 3-uniform hypergraph $K_4^3(-)$
%with 4 vertices and 3 edges is the smallest non-perfect 3-uniform hypergraph. 

{\bf Remark:} The class of graphs ${\cal P}_k$ for which the hypergraphs of cliques %complete subgraphs %independent sets
of size $k$
%as well as the hypergraph of cliques of size $k$
are $H_\omega$-perfect seems an interesting  extension of the class of perfect graphs.
%(Those (?) are graphs without
%induced subgraphs which are $(k-2)$-cones over odd holes or odd antiholes.)
%Of course, starting from a graph $G$ the hypergraph of independent set of $G$ of size $k$
%We note that the classes of graphs for which  
%\subsection {Families of hypergraphs with forbidden induced subcomplexes, and vertex colorings.}

\subsection {$H_\omega$-perfect triangulated spheres}

Recall that $k$-partite $k$-uniform hypergraphs are $H_\omega$-perfect.
The converse is far from being true: indeed, simple hypergraphs
are $H_\omega$-perfect and they can have arbitrary large chromatic number \cite {NesRod79}.
Here is an especially nice class of $H_\omega$-perfect hypergraphs.

\begin {theorem}
  Let $G$ be an $r$-uniform  hypergraph whose edges form a triangulation of an $(r-1)$-dimensional sphere other than $K_{r+1}^r  $.
  Then if $G$ is $H_\omega$-perfect it is
  $r$-partite.   
\end {theorem}

For $r=3$ this is a reformulation of Ore's theorem for planar graphs. The general case
follows from known high dimensional extensions of Ore's theorem
that asserts that if all links of $(r-3)$-faces of an $(r-1)$-dimensional sphere $S$ are even cycles then
the graph of $S$ is $r$-colorable.
See, e.g., \cite {GooOni78} (mainly for $r=4$), \cite {Jos02} (for arbitrary dimensions), %\cite {IzmJos03},
and references in these papers.

\section{Perfect cocycles}
\label {s:co}
\subsection {$H$-perfect cocycles}

\begin {figure}
\centering
\includegraphics[scale=1.0]{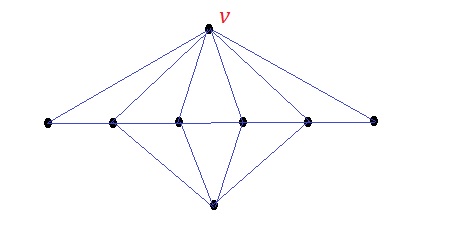}
\caption{A graph whose cocycle is doubly perfect yet not $H$-perfect since $lk_{co(G)}(v)\sim C_7$.}
\label{fig:perfect2}
\end{figure}

Let $G$ be a graph. An unordered triple $\{x,y,z\} \subseteq V(G)$
is {\em $G$-odd} if  $|E(G[\{x,y,z\}])|$ is odd. 
We denote by $co(G)$ the $3$-uniform hypergraph with
vertex set $V(G)$ and such that 
$\{x,y,z \} \in E(co(G))$ if and only $\{x,y,z\}$ is $G$-odd.
Hypergraphs in $co(G)$ are called {\em 3-cocycles} or {\em two-graphs} (\cite {Sei76}).
Note that every cocyle and every complement of a cocycle is
clique friendly. Moreover, by \eqref{tetra} doubly  perfect 3-uniform
hypergraphs are cocycles. 

We start with two observations.

\begin{theorem}\label{cocycleprop}
  Let $G$ be a graph. Then
  \begin{enumerate}
  \item $co(G)^c=co(G^c)$.
  \item For $X \subseteq V(G)$, $co(G)[X]=co(G[X])$.
  \end{enumerate}
\end{theorem}

\begin{theorem}
  \label{doublyperfect}
  A $3$-cocycle is $H$-perfect if and only if it is both
  $H_{\omega}$-perfect and $H_{\alpha}$-perfect.
  In particular, if  a $3$-cocycle is  $H$-perfect, then it is doubly perfect.
  \end{theorem}
\begin{proof}
  Let $F$ be an $H$-perfect $3$-cocycle. Then $F^c$ is $H$-perfect.
  By \ref{cocycleprop} $F^c$ is a cocycle.
  Then both $F$ and $F^c$   are clique friendly. It follows that $F$ and $F^c$ are both
  $H_{\omega}$-perfect, and therefore
  $F$ is both  $H_{\omega}$-perfect and $H_{\alpha}$-perfect.
  Now by \ref{summary} both $F$ and $F^c$ are
  $C_{\omega}$-perfect, and therefore $F$ is doubly perfect.
  \end{proof}

The goal of this section is to describe graphs $G$ for which
$co(G)$ is $H$-perfect, and to
study the larger class of graphs $G$ for which $co(G)$ is doubly perfect.
% (AND MAYBE Cw-PERFECT)
%, or, equivalently doubly perfect. 

For a graph $G$ and a vertex $v \in V(G)$ we denote by $N_G(v)$ the set
of neighbors of $v$, and by $M_G(v)$ the set of non-neighbors of $v$.
Note that $v \not \in N_G(v) \cup M_G(v)$. When there is no danger of
confusion we omit the subscript $G$. Let $v \in V(G)$.
We define the graph $G^+(v)$ as follows.
$V(G^+(v))=V(G) \setminus \{v\}$ and
$xy \in E(G^+(v))$ if and only if one of the following statements holds:
\begin{itemize}
\item $x,y \in N_G(v)$ and $xy \in E(G)$
\item $x,y \in M_G(v)$ and $xy \in E(G)$
\item $x \in N(v)$, $y \in M(v)$ and $xy \not \in E(G)$.
\end{itemize}
Note that $G^+(v)=lk_{co(G)}(v)$.

\begin{comment}
\begin{theorem} \label{linkclique}
  Let $G$ be a graph.
  Then for every $v \in V(G)$ we have
 $ \omega(G^+(v)) \leq \omega(co(G))-1$, and
  $\alpha(G^+(v)) \leq \alpha(co(G)-1$.
    \end{theorem}

\begin{proof}
  By \ref{cocycleprop}.2 it is enough to prove the first statement.
  Let $K$ be a clique in $G^+(v)$. Then by \ref{tetra}
  $K \cup \{v\}$ is a clique in $co(G)$, 
and so $\omega(G^+(v)) \leq \omega(co(G)-1)$.
This proves \ref{linkclique}.
\end{proof}

\begin{theorem} \label{link}
  Let $G$ be a graph such that $co(G)$ is perfect.
Then for every $v \in V(G)$ the graph $G^+(v)$ is 
    $\omega(co(G))-1$-colorable.
    \end{theorem}

\begin{proof}
  This follows immediately from the definition of perfection applied with
  $X=\{v\}$.
\end{proof}

Repeating the proof of \ref{coloringG} we also deduce the converse

{\bf (Maria: of course what I should do is rephrase  \ref{coloringG} in such a way
that this follows, but I didn't want to make changes in Section 1 at this point).}

\begin{theorem} \label{ifflink}
  Let $G$ be a graph such that  for every $v \in V(G)$ the graph $G^+(v)$ is 
  $\omega(co(G))-1$-colorable, and the graph $(G^+(v))^c$ is 
  $\alpha(co(G))-1$-colorable. The $co(G)$ is perfect.
\end{theorem}

\ref{linkclique} allows us to  strengthen \ref{ifflink} further:

\begin{theorem} \label{ifflink+}
  Let $G$ be a graph such that  for every $v \in V(G)$ the graph $G^+(v)$ is 
perfect. The $co(G)$ is perfect.
\end{theorem}

\end{comment}

Next we define a family of graphs that we call ``pre-odd-holes''.
A pair  $(G,v)$ where $G$ is a graph with an even number $\geq 6$ of vertices,
and $v \in V(G)$  is a
{\em pre-odd-hole centered at $v$} if there exists  an even integer $k$ such that
$V(G) \setminus \{v\}$ can be partitioned into $k$ disjoint non-empty subsets
$P_1, \ldots, P_k$
and the edges of $G$ are as follows. 
\begin{itemize}
\item $N_G(v)=\bigcup_{i \text{ even}}P_i$.
\item For every $i$, $G[P_i]$ is a path with vertices $p_1^i, \ldots, p_{n_i}^i$
  in order.
\item If $i \neq j$ and $i=j \mod 2$ then there are no edges between $P_i$ and
  $P_j$.
\item If $i \neq j \mod 2$ and $|i-j| \neq 1 \mod k$, then every vertex of
  $P_i$ is adjacent to every vertex of $P_j$.
\item If $k>2$,  $j=i+1$, or $i=k$ and $j=1$ then $p_{n_i}^i$ is non-adjacent to
  $p_1^j$ and all the other edges between $P_i$ and $P_j$ are present.
\item If $k=2$, then $p_{n_1}^1$ is non-adjacent to $p_1^2$,
  $p_1^1$ is non-adjacent to $p_{n_2}^2$, and all the other edges between
  $P_1$ and $P_2$ are present.
\end{itemize}.

A pair $(G,v)$ is a {\em pre-odd-antihole centered at $v$} if $(G^c,v)$ is a
pre-odd-hole centered at $v$.
We say that a graph $G$ is a {\em pre-odd-hole} if $(G,v)$ is a pre-odd-hole
centered at $v$ for some $v \in V(G)$; a {\em pre-odd-antihole} is defined
similarly.

\begin{comment}
A {\em net} is a graph with vertex set $\{x_1,x_2,x_3,y_1,y_2,y_3\}$ where
$\{x_1,x_2,x_3\}$ is a triangle, and for $i=1,2,3$ $x_i$ is the unique
neighbor of $y_i$. An {\em antinet} is the complement of a net.
\end{comment}
For a graph $G$ and an induced subgraph $H$ of $G$,
we say that $v \not \in V(H)$ is a {\em center} for $H$
if $v$ is complete to $V(H)$, and an {\em anticenter} for $H$ if
$v$ is anticomplete to  $H$.

We say  that $G$ is {\em pure} if
\begin{itemize}
  \item No odd hole of $G$ has a center.
  \item No odd antihole of $G$ has a center.
  \item No odd hole of $G$ has an anticenter.
  \item No odd antihole of $G$ has an anticenter.
  \item No induced subgraph of $G$ is a pre-odd-hole.
  \item No induced subgraph of $G$ is a pre-odd-anti-hole.
\end{itemize}

We prove:
\begin{theorem}\label{perfectcocycle}
  $co(G)$ is $H$-perfect  if and only if $G$ is
  pure.
\end{theorem}

\begin{proof}
  By %\ref{coloringG} and
  the Strong Perfect Graph theorem,
  it is enough to prove that $G$ is pure if and only if for every $v \in V(G)$,
  the graph $G^+(v)$ is Berge.
  
Let us say that $v \in V(G)$ is {\em pure} if all of the following hold.
\begin{enumerate}
  \item $v$ is not a center for an  odd hole of $G$. 
  \item $v$ is not a center for an odd antihole of $G$.
  \item $v$ is not an anticenter for an odd  hole of $G$.
  \item $v$ is not an anticenter for an  odd antihole of $G$. 
  \item $(H,v)$ is not a pre-odd-hole with center $v$ for any induced subgraph
    $H$ of $G$.
  \item $(H,v)$ is not a pre-odd-antihole with center $v$ for any induced
    subgraph $H$ of $G$.
\end{enumerate}    

Cleary $G$ is pure if and only if every vertex of $G$ is pure.

We show that for $v \in V(G)$ the graph $G^+(v)$ is Berge if and only if  $v$
is pure.
  Let $v \in V(G)$, write
  $N=N_G(v)$ and $M=M_G(v)$.
  Suppose $C$  is an odd hole in $G^+(v)$ with vertices
  $c_1,\ldots, c_{2t+1},c_1$ in cyclic order.
  If $V(C) \subseteq N$, then $v$ is a center for $C$, and if
  $V(C) \subseteq M$, then $v$ is a an anticenter for $C$.
Now assume that $C$   meets both $N$ and $M$.
 Let $\{A,B\}=\{N,M\}$.
 Note that
\begin{itemize}
\item  if $x$ and $y$ are two consecutive vertices of $C$, and they are both
 in $A$ or both in $B$, then 
 $xy \in E(G)$,
\item if $x$ and $y$ are two consecutive vertices of $C$, and $x \in A$ and
  $y \in B$, then $xy \not \in E(G)$
\item  if $x,y$ are non-consecutive vertices of $C$, and they are both
 in $A$ or both in $B$, then 
 $xy  \not \in E(G)$,
\item if $x,y$ are non-consecutive vertices of $C$, and  $x \in A$ and $y \in B$,
 then $xy \in E(G)$.
\end{itemize}

A {\em sector} of $C$ is the vertex set of  a maximal path of $C$ that is contained in $M$ or in $N$. Thus for some even integer $k$,  the set
$\{c_1, \ldots, c_{2t+1}\}$  can be partitioned into sectors
$P_1, \ldots, P_k$ where
$\{c_1, \ldots, c_{2t+1}\} \cap N=\bigcup_{i \text{ even}}P_i$.
Now it is easy to check that $G[\{v,c_1, \ldots, c_{2t+1}\}]$  is 
a pre-odd-hole centered at $v$.

Conversly, if $(H,v)$ is a pre-odd-hole centered at $v$, then traversing
the paths $P_1, \ldots, P_k$ (from the defintion of a pre-odd-hole) in order
we obtain and odd hole in $G^+(v)$.

This show that $G^+(v)$ contains an odd hole if and only
if $v$  fails to satisfy
one of the odd-numbered conditions for being pure.
Applying this to $G^c$, we deduce that
$G^+(v)$ contains an odd antihole if and only if $v$ fails to satisfy
one of the even-numbered conditions for being pure.
Thus we showed that $v$ is pure if and only if $G^+(v)$ is Berge.
Since $G$ is pure if and only if every vertex of $G$ is pure, this proves~\ref{perfectcocycle}.
\end{proof}

In view of \ref{doublyperfect} we deduce
\begin{theorem}\label{doublyperfectcocycle}
  $G$ is pure if and only if $co(G)$ is both $H_{\omega}$-perfect and
  $H_{\alpha}$-perfect. In particular, if
   $G$ is pure, then   $co(G)$ is doubly perfect.
 \end{theorem}

\subsection {Doubly-perfect cocycles}

We present many examples of graphs $G$ such that $co(G)$ is not doubly perfect.
The class of graphs $G$ for which $co(G)$ is doubly perfect is closed under induced subgraphs
and hence it can be described in terms of forbidden induced subgraphs. However, our examples
suggest that such a description could be out of reach.

Start with a triangle-free graph $H$ with  chromatic number  $\chi (H) >3$.
Next, add a new vertex $v$ adjacent to an arbitrary subset $A \subset V(H)$. Consider now the graph
$G$ obtained from $H+v$ by the following operation: for every edge $e=(x,y)$ where
$x \in A$ and $y \in V(H) \backslash A$, $e \in E(G)$ iff $e \notin E(H)$. For all other edges
$e \in E(G)$ iff $e \in E(H)$. (In other words, we switch between edges and nonedges between $A$ and $V(H)
\backslash A$.)

\begin {theorem}\label{triangle-free}
  $co(G)$ is not $C_\omega$-perfect.
\end {theorem}
\begin {proof}
We need two claims.  
\begin {theorem}\label{triangle-free2}
Let $v \in V(G)$. Then $lk_{co(G)}(v) =H$
\end {theorem}
\begin {proof}
A triple $\{v,a,b\}$ belongs to $co(G)$ if it has an odd number of edges from $G$.
If $a$ and $b$ are both in $A$ or in $V(H) \backslash A$ this occurs iff $ab$ is an edge in $H$.
If $a \in A$ and $b \in V(H) \backslash A$ this occurs iff $ab$ is not an edge of $G$ and hence is an edge in $H$.
\end {proof}
\begin {theorem}\label{triangle-free3}
$\omega (co (G))=4$.
\end {theorem}
\begin {proof}
Since $lk_{co(G)}(v) =H$, there is not even $K_4^3$ in $co(G)$ that includes the vertex $v$. $co (G)$ and $co(H)$ restricted to
all other vertices coincide and it is easy to verify that $co(H)$ contains no $K_5^3$ when $H$ is triangle-free.
\end {proof}
To conclude the proof of \ref {triangle-free}
we note that a proper coloring of pairs of vertices of $co (G)$ with $t$ colors restricts to a proper vertex coloring of $H$
with $t$ colors, and since
  $\chi(H) >3  =  \omega(co(G))-1$, we deduce that $co(G)$ is not $C_\omega$-perfect.
\end {proof}

\subsection {Simonyi's characterization of entropy splitting hyperhraphs}

Now we  describe  a notion of perfect hypergraphs %introduced by Gabor Simonyi in 1998
which is closely
related to (and yet interestingly different from) our notion of doubly-perfect hypergraphs.
In \cite {CKLMS90} Csisz\'ar, K\"orner, Lov\'asz, Marton, and Simonyi gave a characterization of perfect graphs in
terms of the equality case of certain subadditivity inequality (by K\"orner) involving graph entropy.

Gabor Simonyi \cite {Sim95} studied cases of equality for an extended entropic inequality  for $k$-uniform hypergraphs. This led to the
the class of entropy splitting hypergraphs giving an extension of the notion of perfectness to uniform hypergraphs.
Both these papers are related to K\"orner's important notion of
graph entropy \cite {Kor73} and subsequent works by K\"orner and Longo, and K\"orner and Marton.

%What I found shows some connections to the conditions you write about in your new paper in connection to doubly perfect hypergraphs.
%What I found is this:
As proved by Simonyi, for $k>3$ only complete and edgeless hypergraphs have the entropy splitting property.
For $k=3$ entropy-splitting hypergraphs is a restricted class of 3-uniform hypergraphs and
below are several equivalent characterizations of this class:

\begin {enumerate}
    
  \item  On every four vertices they have an even number of edges and, in addition, they do not contain a special
    hypergraph on 5 vertices as an induced subhypergraph. This special hypergraph is (in our language)  $co(C_5)$,
    the cocycle defined by the five length cycle.

  \item These 3-uniform hypergraphs can be obtained by starting with a single edge on three vertices
    and applying two operations (any number of times in any order):

    a) taking the complementary 3-uniform hypergraph

    b) duplicating a vertex.

    Here, duplicating a vertex $v$ consists of introducing a new vertex $v'$
    that appears in an edge $v'xy$ iff $vxy$ is also an edge. (The two vertices $v$ and $v'$ do not appear together in any edge.)
    %This means that all triples containing both of them become edges when a complementation is also applied.)

  \item The class of cocycles of cographs. (We recall that cographs are graphs that can be obtained from a single vertex graph
    by repeated applications of disjoint union and taking complements. Equivalently they are the class of graphs
    with no induced path on four vertices.)
    
    \item Entropy splitting 3-uniform hypergraphs have a ``leaf-pattern'' representation defined as follows:
      Given a tree with all its inner (non-leaf) vertices labeled with 0 or 1 (in an arbitrary manner),
      the leaf-pattern of this labeled tree is the following 3-uniform hypergraph.
      Its vertices are the leaves of the tree and three leaves, $x,y,z$ form an edge iff the unique common
      point of the three paths $x-y, y-z, x-z$ is labeled with 1.
\end {enumerate}
      
The equivalence of the classes given by the first and second items requires work
%1 and 2 are the same and on the way I had to prove that, too, but
%when lecturing about it on a Princeton workshop
and it turned out that this equivalence goes back to a 1984 paper by Gurvich \cite {Gur84}. 

It follows that entropy-splitting 3-uniform hypergraphs form a subclass of $H$-perfect 3-uniform cocycles
(and hence also of doubly perfect 3-uniform hypergraphs). It is easy to see that cographs are pure. The first four
obstruction to purity contain holes or antiholes, and therefore contain paths on four vertices.
The last two require a little more analysis, but it is still true
that they contain four-verex paths.
$C_5$ is pure simply because all obstruction to purity have at least 6 vertices.
Therefore the 5-vertex hypergraph $co(C_5)$ is an $H$-perfect cocycle and it is not an entropy splitting hypergaph. 
%still $C_{\omega}$-perfect (and also $C_{\alpha}$-perfect
%as it is self-complementary), and 
%the class of doubly perfect hypergraphs is richer.
%%I hope I am not mistaken when claiming that my class is contained in your class:

{\bf Remark:} Gabor
Simonyi pointed out to us also a direct inductive argument that entropy-splitting hypergraphs are doubly perfect.
%I think it can be proven by induction using the above description 2 (or 3), though it is
%not entirely trivial (because of the possible complementations).
%It seems to me that
%and suggested that his class coincides with what one would get for doubly perfectness if exchanging quantifiers
%and require the {\it same} proper coloring to work for the link of $X$ for all X with $|X|<k-1$
%in the definition of $C_{\omega}$-perfectness.
%(He noted that $co(C_5)$ does not satisfy the conditions any more.)

%\section {$\chi$-boundedness, $\chi^c$-boundedness and other notions of perfectness}
\section {Connections, problems, and other notions of perfectness}
\label {s:o}

%{\bf Gil: this is mainly for us but we can mention some of it}

\subsection {Vertex colorings and $\chi$-boundedness}

The families of hypergraphs described in this paper are closed under induced subhypergraphs and therefore can be described
in terms of forbidden induced subhypergraphs. The study of families of
graphs described in terms of forbidden induced subgraphs is
a rich part of graph theory and extensions for uniform hypergraphs are of interest.
In our definitions we replaced vertex coloring for perfect graphs with proper
colorings of sets of $(k-1)$ vertices but it is of interest to also examine vertex colorings. 
%\subsection {\chi$-bounded hypergraphs}

  A family of $k$-uniform hypergraphs is $\chi$-bounded if for some
  function $f$ the vertices of every hypergraph $G$
  in the family  can be covered by $f(\omega (G))$ independent sets.
  A family of $k$-uniform hypergraphs is $\chi^c$-bounded if the family of its complements are $\chi$-bounded
  or , in other words, if for some
  function $f$ the vertices of every hypergraph $G$
  in the family  can be covered by $f(\alpha (G))$ cliques.
  For graphs ($k=2$) this notion was studied starting with works of Gy\'arf\'as and others,
  see \cite {ScoSey20} for a recent survey. For example, it is known
  %\cite { }
  that the class of graphs with no odd holes is $\chi$-bounded.

  {\bf Problem:} Are $H_\omega$-perfect (or even $C_\omega$-perfect) $k$-uniform hypergraphs $\chi^c$-bounded?
  (Equivalent formulation: Are $H_\alpha$-perfect hypergraphs $\chi$-bounded?)

As we already mentioned, $H_\omega$-perfect hypergraphs are not $\chi$-bounded since
simple hypergraphs that can have arbitrary large chromatic number are $H_\omega$-perfect.

It is easy to see (directly) that doubly-perfect 3-hypergraphs are $\chi$-bounded with $f(m)=m-1$.
Indeed if $\omega (G)=t$ then for a vertex $v$ $\omega (lk_G(v))=t-1$ and therefore the vertices of $lk_G(v)$
can be covered by  $t-1$ independent sets. Now, since $G$ is a cocycle
these independent sets remain independent in $G$ and
adding $v$ to
one of them describes $V(G)$ as the union of $t-1$ independent set.
The class of $\chi$-bounded $k$-uniform hypergraphs
with $f(m)=m-k+2$ seems another interesting extension of perfect graphs.
The strongest form of $\chi$-boundedness that we can have for $k$-uniform hypergraphs is with
%and so is the class of $\chi$-bounded $k$-uniform hypergraphs with
$f(n)=\lceil n/(k-1) \rceil$,
and the restricted class of $\chi$-bounded $k$-uniform hypergraphs
with $f(n)=\lceil n/(k-1) \rceil$ is yet another extension
of perfect graphs worthy of study.

\subsection {Analogs for chordal graphs and for interval graphs}

A different avenue for perfect hypergraphs could start with various high dimensional
extensions from the literature for chordal graphs. One definition is that a $k$-uniform
hypergraph $G$ is ``chordal''
if its clique complex $C(G)$ is $(k-1)$-Leray, meaning that $H_i(K)=0$ for every $i\ge k-1$ and
every induced subcomplex $K$ of $C(G)$. (Here, the clique complex $C(G)$ of a hypergraph  $G$ 
is the simplicial complex on the vertices of $G$ whose faces correspond to sets of vertices that span a complete hypergraph.)
This definition of chordality coincided with the
definition of chordal graphs for $k=2$, namely
the homological condition simply asserts that the graph has no induced cycles with more than three edges.
Certain refinements of this notion were considered in \cite {ANS16,Adi17}. A very different notion
of ``chordal hypergraphs'' was defined by Voloshin in his book \cite {Vol09} Chapter 8.

  %An interesting question is that of $\chi$-boundedness for $G$ and for its complement.
  It is known, \cite {AKMM02} that
  if $G$ is a $k$-uniform chordal hypergraph (according to the definition above)
  then it is $\chi^c$-bounded,
  in other words if $\alpha=\alpha (G)$ then $G$ can be covered by $f(\alpha,k)$ cliques.

An interesting larger class of hypergraphs extending the class of graphs with no odd holes
would allow $(k-1)$-dimensional homological cycles provided they are $k$-partite.
We do not know if $\chi^c$-boundedness still holds.
(Imposing this condition both for $G$ and its complement leads again to the class
of perfect graphs for $k=2$.)

A special case of chordal $k$-uniform hypergraphs (which motivated their definition)
are those represented by collections of convex sets in $R^{k-1}$. Given such a collection we consider the hypergraph where
vertices correspond to sets of the collection, and edges correspond to $k$ sets with non-empty intersection. 
When $k=2$ we obtain the class of interval graphs.

\subsection {The Hadwiger--Debrunner property}

%  Another property worth considering (for simplicity for $k=3$) is the following:
%  Consider the class of $k$-uniform hypergraphs with the following property (that we call the HD-property):  
%  If every set of $p$ vertices contains a complete subhypergraph of size $q$
%  where $q \ge k$ and

Another notion worth considering %(for simplicity for $k=3$)
is the following:
Consider the class of $k$-uniform hypergraphs $G$ with the following property (that we call the $HD_r$-property):  
If $H$ is an induced subhypegraph of $G$ and if $p$ and $q$, $V(H) \ge p > q \ge k$, are integers satisfying
  \begin {equation}
      p(r-1)<(q-1)r,
     \end {equation}
\noindent
and if every set of $p$ vertices, contains a complete subhypergraph of size $q$
%where $q \ge k$ and
then $V(G)$ can be covered by $p-q+1$ cliques.

 Wagner conjectured (see \cite {Eck03}) that every $k$-uniform hypergraph has the $HD_k$ property, and we are mainly
  interested in the case that $r=k-1$.
  Perfect graphs are graphs with the $H_1$-property.
  A proof of the $HD_{k-1}$-property
  for $k$-uniform hypergraphs associated with families of convex sets in $R^{k-1}$ goes back to Hadwiger and Debrunner \cite {HD57}.  

\subsection {Voloshin's C-perfectness}

A different notion of perfectness for hypergraphs was pioneered by Voloshin \cite {Vol95,Vol02}
and further studied by Bujt\'as and Tuza \cite {BujTuz09}. Voloshin considered colorings of hypergraphs $G$ where
no edge is multicolored (or rainbow). He defined $\bar \chi (G)$ as the maximum number of colors in such
a coloring. Clearly, $\bar \chi (G) \leq \alpha (G)$,  
and $G$ is perfect according to Voloshin if for every induced subhypergraph $H$,
$\bar \chi (H)=\alpha (H)$.  

\subsection {Is there a homological description of  perfect graphs?}

We mentioned that chordal graphs have a simple homological description: $G$ is chordal if and only if for every induced subgraph $H$ of $G$, $H_i(C(H))=0$, for every $i>0$, where $C(H)$ is the clique complex of $H$.  
It is an interesting question if a similar homological definition exists for perfect graphs.  
  
%{\bf Remark:} 
%We note that if we replace in the inequality $k$ by $k+1$ the property
%holds for all $k$-uniform hypergraphs.
%  \begin {theorem} $\omega$-perfect $k$-uniform hypergraphs have the
%    $HD_{k-1}$-property.
%    $HD_{2}$-property.
%  \end {theorem}

\subsection *{Acknowledgment} We are thankful to Gabor Simonyi, Vitaly Voloshin and an anonymous
referee for helpful remarks. We thank
the referee especially for detecting several mistakes in an earlier version of the paper.

\end{document}